\documentclass[10pt,reqno]{amsart}
\usepackage{amsmath,multicol}
\usepackage{amsthm}
\usepackage{amssymb}
\usepackage{amsfonts}
\usepackage{subcaption}
\usepackage{graphicx}
\usepackage{latexsym}
\usepackage{cite, url, xcolor}
\usepackage{stmaryrd}
\usepackage{hyperref}
\usepackage{enumitem}
    \setenumerate{itemsep=5pt} 
    \setitemize{itemsep=5pt} 
    
\usepackage[font=Large]{caption}

\usepackage[font=footnotesize,labelfont=bf]{caption}

\usepackage[sc]{mathpazo}
\usepackage[T1]{fontenc}

\usepackage{tikz,amsmath,amsfonts}
\usetikzlibrary{positioning}
\usetikzlibrary{decorations.pathreplacing}

\newcommand{\NN}{\mathbb{N}}
\newcommand{\ZZ}{\mathbb{Z}}
\newcommand{\RR}{\mathbb{R}}
\newcommand{\CC}{\mathbb{C}}

\newcommand{\Le}{\mathsf{L}}

\newcommand{\lcm}{\operatorname{lcm}}
\renewcommand{\vec}[1]{{\bf #1}}
\providecommand{\multi}[1]{\llbracket #1 \rrbracket}

\renewcommand{\o}[1]{\overline{#1}}

\renewcommand{\pmod}[1]{\,(\operatorname{mod} #1)}

\newcommand{\stirling}[2]{\genfrac\{\}{0pt}{}{#1}{#2}}

\renewcommand\>{\rangle}
\newcommand\<{\langle}

\theoremstyle{plain}
\newtheorem{Theorem}[equation]{Theorem}
\newtheorem{Lemma}[equation]{Lemma}

\theoremstyle{definition}

\newtheorem{Example}[equation]{Example}

\allowdisplaybreaks
\begin{document}
\title[Factorization length distribution for affine semigroups III]{Factorization length distribution for affine semigroups~III: modular equidistribution for numerical semigroups with arbitrarily many generators}

\author[S.~Garcia]{Stephan Ramon Garcia}
\address{Department of Mathematics, Pomona College, 610 N. College Ave., Claremont, CA 91711} 
\email{stephan.garcia@pomona.edu}
\urladdr{\url{http://pages.pomona.edu/~sg064747}}

\author[M.~Omar]{Mohamed Omar}
\address{Department of Mathematics, Harvey Mudd College, 301 Platt Blvd., Claremont, CA 91711}
\email{omar@g.hmc.edu}
\urladdr{\url{www.math.hmc.edu/~omar}}

\author[C.~O'Neill]{Christopher O'Neill}
\address{Mathematics Department, San Diego State University, San Diego, CA 92182}
\email{cdoneill@sdsu.edu}
\urladdr{\url{https://cdoneill.sdsu.edu/}}

\author[T.~Wesley]{Timothy Wesley}
\address{Pomona College, Claremont, CA}
\email{tgwa2017@pomona.edu}

\thanks{First author partially supported by NSF grant DMS-1800123. Thanks also to the anonymous referee for their careful reading of the paper.}


\begin{abstract}
For numerical semigroups with a specified list of (not necessarily minimal) generators, we describe the asymptotic distribution of factorization lengths with respect to an arbitrary modulus.  In particular, we prove that the factorization lengths are equidistributed across all congruence classes that are not trivially ruled out by modular considerations.
\end{abstract}

\maketitle

\section{Introduction}\label{sec:intro}

In what follows, we let $\NN = \{0,1,2,\ldots\}$ and denote the cardinality of a set $X$ by $|X|$.
A~\emph{numerical semigroup} $S \subset \NN$ is an additive subsemigroup containing $0$.  Each numerical semigroup $S$ admits a finite generating set, and we write 
\begin{equation*}
S = \<n_1, n_2, \ldots, n_k\> = \{a_1n_1 + a_2 n_2 + \cdots + a_kn_k \,:\, a_1,a_2, \ldots, a_k \in \NN\}
\end{equation*}
for the numerical semigroup generated by the distinct positive $n_1 < n_2 < \cdots < n_k$.
Throughout this document, we always assume $S$ has finite complement in $\NN$, or, equivalently, that $\gcd(n_1,n_2, \ldots, n_k) = 1$.  However, we do not assume that $n_1, n_2, \ldots, n_k$ form the unique minimal generating set of $S$ under containment~\cite{NSBook}.  

A \emph{factorization} of $n \in S$ is an expression
\begin{equation*}
n = a_1n_1 + a_2 n_2 + \cdots + a_kn_k
\end{equation*}
of $n$ as a sum of generators of $S$, denoted by the $k$-tuple 
$\vec{a} = (a_1, a_2, \ldots, a_k) \in \NN^k$.  
The \emph{length} of the factorization $\vec{a}$ is 
\begin{equation*}
\|\vec{a}\| = a_1 + a_2 + \cdots + a_k.
\end{equation*}
The \emph{length multiset of $n$}, denoted $\Le\multi{n}$, is the multiset 
with a copy of $\|\vec{a}\|$ for each factorization $\vec{a}$ of $n$.  Recall that a \emph{multiset} 
is a set in which repetition is taken into account; that is, its elements can occur multiple times.  In particular, the cardinality 
$|\Le \multi{n}|$ equals the number of factorizations of $n$.  

Factorizations and their lengths have been studied extensively in numerous contexts, including factorization theory~\cite{nonuniq,GH92,structurethm}, additive combinatorics~\cite{krulldeltaset,numericalfactorsurvey}, discrete optimization~\cite{pisinger1998knapsack,de2013algebraic}, commutative and non-commutative algebra~\cite{modulessurvey,noncommutativefactor}, and algebraic geometry~\cite{abhyankar1967local,barucci1997maximality}.  
Until recently, results concerning the multiplicities of factorization lengths have been surprisingly absent from the literature.  Initiating the study of $\Le\multi{n}$ was ~\cite{lengthdistribution1}, wherein a closed form for the limiting distribution is obtained for three-generator numerical semigroups via careful combinatorial arguments.  In a sequel paper~\cite{GOOY}, measure theory and algebraic combinatorics are used to characterize the distribution for arbitrary numerical semigroups.  
This present manuscript, the third in this series, examines the spread of $\Le\multi{n}$ across congruence classes modulo a fixed positive integer $N$.  More precisely, given $n \in S$, $N \ge 1$, and $i \in \{0, \ldots, N-1\}$, we study the distribution of $\Le\multi{n} \cap (i + N\ZZ)$.  

Before we introduce the main theorem, let us briefly recall the constant
\begin{equation*}
\delta = \gcd(n_2 - n_1,\, n_3-n_2,\ldots, \,n_k - n_{k-1})
\end{equation*}
and its relation to factorization lengths.  Given $n \in S$, all lengths $\ell_1, \ell_2 \in \Le\multi{n}$ must satisfy $\ell_1 \equiv \ell_2 \pmod \delta$ (in fact, $\delta$ is the largest integer with this property by \cite[Prop.~2.9]{bowles2006delta}).  In particular, the intersection $\Le\multi{n} \cap (i + N\ZZ)$ is sometimes empty, even for arbitrarily large $n$.  

Let us consider an example.  For the numerical semigroup $S = \<7, 19, 25, 31\>$, every factorization length of $n = 434$ is congruent modulo $\delta = 6$.  More specifically, since $n = 62 \, n_1$, every length in $\Le\multi{n}$ is congruent to $2$ modulo $6$.  As such, for~$N = 4$, only the intersections $\Le\multi{n} \cap 4\ZZ$ and $\Le\multi{n} \cap (2 + 4\ZZ)$ can be nonempty.  If, on the other hand, we had chosen $n$ odd, then every element of $\Le\multi{n}$ must be congruent to either $1$ or $3$ modulo $4$.  
Among other things, Theorem~\ref{Theorem:Moments} below implies that in this example, for $n$ large and odd, the multisets $\Le\multi{n} \cap (1 + N\ZZ)$ and $\Le\multi{n} \cap (3 + N\ZZ)$ have identical distributions.  

To state Theorem~\ref{Theorem:Moments}, we require some algebraic terminology.
The \emph{complete homogeneous symmetric polynomial} of degree $p$ in the
$k$ variables $x_1, x_2, \ldots, x_k$ is
\begin{equation*}
h_p(x_1,x_2,\ldots,x_k) \quad=\!\! 
\sum_{1 \leq \alpha_1 \leq \cdots \leq \alpha_{p} \leq k} x_{\alpha_1} x_{\alpha_2}\cdots x_{\alpha_p},
\end{equation*}
the sum of all degree $p$ monomials in $x_1,x_2,\ldots,x_k$.  A \emph{quasipolynomial} of degree $d$ is a function $f:\ZZ\to\CC$ of the form
\begin{equation*}
f(n) = c_d(n) n^d + c_{d-1}(n) n^{d-1} + \cdots + c_1(n) n + c_0(n),
\end{equation*}
in which the coefficients $c_1(n), c_2(n),\ldots, c_d(n)$
are periodic functions of $n \in \ZZ$ \cite{continuousdiscretely}.  A \emph{quasirational function} is a quotient of two quasipolynomials.

\begin{Theorem}\label{Theorem:Moments}
Let $S = \<n_1,n_2,\ldots,n_k\>$ with $\gcd(n_1,n_2,\ldots,n_k)=1$ and define $\delta$ as above.  Fix $N \in \NN$ and let $m = \gcd(\delta, N)$.  
Then 
\begin{equation*}
\sum_{\substack{\ell \in \Le\multi{n} \\ \ell \equiv i \pmod{N}}} \!\!\!\!\!\!\!\! \ell^p \,\, =\,\,
\begin{cases}
\dfrac{p! \,m\, h_p ( \frac{1}{n_1}, \frac{1}{n_2}, \ldots, \frac{1}{n_k} )}{N(k + p - 1)! (n_1 \cdots n_k)}  n^{k+p-1} + w_i(n)
  & \text{if $n \equiv i n_1 \pmod{ m }$}, \\[10pt]
0 &\text{if $n \not \equiv i n_1 \pmod{ m}$},
\end{cases}
\end{equation*}
in which $w_i(n)$ is a quasipolynomial of degree at most $k+p-2$ whose coefficients are rational valued and have period dividing 
$N\lcm(n_1,n_2,\ldots,n_k)$.
\end{Theorem}

The choice of $m$ in Theorem~\ref{Theorem:Moments} implies
\begin{equation}\label{eq:AllSameMod}
n_1 \equiv n_2 \equiv \cdots \equiv n_k \pmod{m}.
\end{equation}
Together with the assumption $\gcd(n_1,n_2,\ldots,n_k)=1$, it follows that $n_1,n_2,\ldots,n_k$ are invertible modulo $m$.
Since $n_1$ is invertible modulo $m$,
there is a unique $i$ modulo $m$ such that $n \equiv i n_1 \pmod{ m}$.  Because $m \mid N$,
there are $N/m$ distinct values of $i$ modulo $N$ for which this occurs.
As such, Theorem~\ref{Theorem:Moments} yields (depending upon the parameters involved) $m/N$ or $0$ times the corresponding result from \cite[Thm.~2]{GOOY} on $\sum_{\ell \in \Le\multi{n}} \ell^p$.

Define a sequence of probability measures on $[0,1]$ by
\begin{equation*}
\nu_n = \frac{1}{| \Le \multi{n}|}\sum_{\ell \in \Le\multi{n}} \delta_{\ell/n},
\end{equation*}
in which $\delta_x$ denotes the point mass at $x$ (not to be confused with the number $\delta$ defined
earlier; both notations are standard and unavoidable).  As shown in \cite[Thm.~1]{GOOY},
these measures converge weakly to an absolutely continuous measure whose probability density
function is a certain Curry--Schoenberg B-spline.  Theorem~\ref{Theorem:Moments}
permits us to obtain a result analogous to \cite[Thm.~1]{GOOY}, with all results scaled by $m/N$ or $0$ depending on the relevant congruence class.

\begin{Theorem}\label{Theorem:Probability}
Let $S = \<n_1,n_2,\ldots,n_k\>$, in which $k \geq 3$ and $\gcd(n_1,n_2,\ldots,n_k)=1$, and define $\delta$ as above.  Fix $N \in \NN$, and let $m = \gcd(\delta, N)$.  
\begin{enumerate}
\item For real $\alpha < \beta$,
\begin{align*}
\lim_{n\to\infty} \frac{ |\{ \ell \in \Le\multi{n} : \ell \equiv i \pmod{N}, \, \,  \ell \in[\alpha n, \beta n]\} |}{ | \Le\multi{n} |} 
=
\begin{cases}
\displaystyle \frac{m}{N}  \int_{\alpha}^{\beta} F(t)\,dt & \text{if $n \equiv i n_1 \pmod{ m }$},\\[8pt]
 0&\text{if $n \not \equiv i n_1 \pmod{ m}$},
\end{cases}
\end{align*}
where $F:\RR\to\RR$ is the probability density function 
\begin{equation*}
F(x) :=\frac{(k-1)n_1n_2\cdots n_k}{2}\sum_{r=1}^k \frac{|1-n_rx|(1-n_r x)^{k-3}}{\prod_{j\neq r}(n_j-n_r)}.
\end{equation*}
The support of $F$ is $\big[\frac{1}{n_k}, \frac{1}{n_1}\big]$.  

\item For any continuous function $g:(0,1)\to\CC$,
\begin{equation*}
\lim_{n\to\infty}  \frac{1}{ |\Le\multi{n}|} \sum_{\substack{\ell \in \Le\multi{n} \\ \ell \equiv i \pmod{N}}} 
 g\bigg(\frac{\ell}{n}\bigg) 
 =
\begin{cases}
\displaystyle \frac{m}{N}  \int_0^1 g(t) F(t)\,dt & \text{if $n \equiv i n_1 \pmod{ m }$},\\[8pt]
 0&\text{if $n \not \equiv i n_1 \pmod{ m}$},
\end{cases}
\end{equation*}
\end{enumerate}
\end{Theorem}

Before proceeding to the proof of Theorem \ref{Theorem:Moments} in Section \ref{Section:Proof}, we first examine some applications and examples in Section \ref{Section:Applications}.  We conclude in 
Section \ref{Section:Conclusion} an open question that provides a possible avenue for future research.

\section{Applications and Examples}\label{Section:Applications}

Theorem \ref{Theorem:Moments} is remarkable since it applies to all numerical semigroups
and all moduli.  Consequently, there is a lot to explore and comment on.

\begin{Example}
The generality of~\cite[Thm.~1 and~2]{GOOY} yielded asymptotic descriptions of numerous statistics of $\Le\multi{n}$, including the mean, median, mode, variance, standard deviation, and skewness, as well as the harmonic and geometric means; see \cite[Section~2.1]{GOOY} for precise statements.  As an example, the mean $m_1(n)$ of $\Le\multi{n}$ was shown to satisfy
\begin{equation*}
m_1(n) = \frac{1}{|\Le\multi{n}|}\sum_{\ell \in \Le\multi{n}} \ell
\,\sim\,
\frac{n}{k}\bigg(\frac{1}{n_1} + \cdots + \frac{1}{n_k}\bigg)
\end{equation*}
as $n \to \infty$.  
As a consequence of Theorems~\ref{Theorem:Moments} and~\ref{Theorem:Probability}, we immediately obtain analogous asymptotic descriptions of each statistic for the intersection of $\Le\multi{n}$ with any fixed congruence class modulo $N$.  In particular, $\Le\multi{n} \cap (i + N\ZZ)$ is empty unless $n \equiv in_1 \pmod m$, in which case for $n$ large,
\begin{equation*}
\frac{1}{|\Le\multi{n}|}\sum_{\substack{\ell \in \Le\multi{n} \\ \ell \equiv i \pmod{N}}} \ell
\,\sim\,
\frac{m}{N} \cdot \frac{n}{k}\bigg(\frac{1}{n_1} + \cdots + \frac{1}{n_k}\bigg).
\end{equation*}
\end{Example}

\begin{Example}
The simplest case of Theorem~\ref{Theorem:Moments} is $N = 1$, which forces $m=1$.  Since $\ell \equiv i \pmod{1}$ for all 
integers $\ell$ and $i$, we obtain \cite[Thm.~2]{GOOY}:
\begin{equation}\label{eq:GOOY}
\sum_{\ell \in \Le\multi{n}} \ell^p = 
 \dfrac{p! \, h_p ( \frac{1}{n_1}, \frac{1}{n_2}, \ldots, \frac{1}{n_k} )}{(k + p - 1)! (n_1 n_2 \cdots n_k)}  n^{k+p-1} + w(n),
\end{equation}
in which $w(n)$ is a quasipolynomial of degree at most $k+p-2$ whose coefficients are rational valued and have period dividing $\lcm(n_1,n_2,\ldots,n_k)$.
\end{Example}

\begin{table}
\small
\begin{equation*}
\begin{array}{c|c|c}
\text{Residue} & \text{Count} & \text{Proportion}\\
\hline
0 \pmod{ 2 } &   233 &   0.5011 \\
1 \pmod{ 2 } &   232 &   0.4989 \\
\hline
0 \pmod{ 3 } &   155 &   0.3333 \\
1 \pmod{ 3 } &   155 &   0.3333 \\
2 \pmod{ 3 } &   155 &   0.3333\\
\hline
0 \pmod{ 4 } &   115 &   0.2473 \\
1 \pmod{ 4 } &   116 &   0.2495 \\
2 \pmod{ 4 } &   118 &   0.2538 \\
3 \pmod{ 4 } &   116 &   0.2496\\
\hline
0 \pmod{ 5 } &   94 &   0.2022\\
1 \pmod{ 5 } &   93 &   0.2000\\
2 \pmod{ 5 } &   93 &   0.2000\\
3 \pmod{ 5 } &   92 &   0.1978\\
4 \pmod{ 5 } &   93 &   0.2000\\
\hline
0 \pmod{ 6 } &   77 &   0.1656\\
1 \pmod{ 6 } &   77 &   0.1656\\
2 \pmod{ 6 } &   78 &   0.1678\\
3 \pmod{ 6 } &   78 &   0.1678\\
4 \pmod{ 6 } &   78 &   0.1677\\
5 \pmod{ 6 } &   77 &   0.1656\\
\end{array}
\qquad
\begin{array}{c@{}c|c|c}
\text{Residue} && \text{Count} & \text{Proportion}\\
\hline
0 \pmod{ 7 } &&   72 &   0.1548\\
1 \pmod{ 7 } &&   73 &   0.1570\\
2 \pmod{ 7 } &&   59 &   0.1269\\
3 \pmod{ 7 } &&   62 &   0.1333\\
4 \pmod{ 7 } &&   64 &   0.1376\\
5 \pmod{ 7 } &&   66 &   0.1419\\
6 \pmod{ 7 } &&   69 &   0.1484\\
\hline
0 \pmod{ 8 } &&   58 &  0.1247\\
1 \pmod{ 8 } &&   58 &  0.1247 \\
2 \pmod{ 8 } &&   59 &   0.1269\\
3 \pmod{ 8 } &&   58 &   0.1247\\
4 \pmod{ 8 } &&   57 &   0.1226\\
5 \pmod{ 8 } &&   58 &   0.1247\\
6 \pmod{ 8 } &&   59 &   0.1269\\
7 \pmod{ 8 } &&   58 &   0.1247\\
\hline
\\
\\
\\
\\
\\
\end{array}
\end{equation*}
\caption{The McNugget semigroup $\<6,9,20\>$ has $\delta = 1$.
For each modulus $N$ and residue $i \pmod{N}$, the proportion of lengths $\ell \in \Le\multi{n}$ with $\ell \equiv i \pmod{N}$ tends to $1/N$.  Even for $n=1{,}000$, as depicted above, this behavior is evident.}
\label{Table:McNugget}
\end{table}

\begin{table}
\small
\begin{equation*}
\begin{array}{c|c|c}
\text{Residue} & \text{Count} & \text{Proportion}\\
\hline
0 \pmod{ 2 } &   14500 &   1.0000 \\
1 \pmod{ 2 } &   0     &   0.0000 \\
\hline
0 \pmod{ 3 } &   0     &   0.0000 \\
1 \pmod{ 3 } &   14500 &   1.0000 \\
2 \pmod{ 3 } &   0     &   0.0000 \\
\hline
0 \pmod{ 4 } &   7349  &   0.5068 \\
1 \pmod{ 4 } &   0     &   0.0000 \\
2 \pmod{ 4 } &   7151  &   0.4932 \\
3 \pmod{ 4 } &   0     &   0.0000 \\
\hline
0 \pmod{ 5 } &   2890  &   0.1993 \\
1 \pmod{ 5 } &   2910  &   0.2007 \\
2 \pmod{ 5 } &   2909  &   0.2006 \\
3 \pmod{ 5 } &   2888  &   0.1992 \\
4 \pmod{ 5 } &   2903  &   0.2002 \\
\hline
0 \pmod{ 6 } &   0     &   0.0000 \\
1 \pmod{ 6 } &   0     &   0.0000 \\
2 \pmod{ 6 } &   0     &   0.0000 \\
3 \pmod{ 6 } &   0     &   0.0000 \\
4 \pmod{ 6 } &   14500 &   1.0000 \\
5 \pmod{ 6 } &   0     &   0.0000 \\
\end{array}
\qquad
\begin{array}{c@{}c|c|c}
\text{Residue} && \text{Count} & \text{Proportion}\\
\hline
0 \pmod{ 7 } &&   2096  &   0.1446 \\
1 \pmod{ 7 } &&   2094  &   0.1444 \\
2 \pmod{ 7 } &&   2045  &   0.1411 \\
3 \pmod{ 7 } &&   2031  &   0.1401 \\
4 \pmod{ 7 } &&   2066  &   0.1424 \\
5 \pmod{ 7 } &&   2084  &   0.1437 \\
6 \pmod{ 7 } &&   2084  &   0.1437 \\
\hline
0 \pmod{ 8 } &&   3682  &   0.2539 \\
1 \pmod{ 8 } &&   0     &   0.0000 \\
2 \pmod{ 8 } &&   3578  &   0.2468 \\
3 \pmod{ 8 } &&   0     &   0.0000 \\
4 \pmod{ 8 } &&   3667  &   0.2529 \\
5 \pmod{ 8 } &&   0     &   0.0000 \\
6 \pmod{ 8 } &&   3573  &   0.2464 \\
7 \pmod{ 8 } &&   0     &   0.0000 \\
\hline
\\
\\
\\
\\
\\
\end{array}
\end{equation*}
\caption{The semigroup $\<17, 29, 47, 65\>$ has $\delta = 6$.
The possible congruence classes for $\ell \in \Le\multi{n}$ depends on $N$, but those residues that can be attained are attained evenly as $n \to \infty$.  Given above are the counts for $n=5{,}000$.}
\label{Table:BiggerDelta}
\end{table}

\begin{Example}
Fix a modulus $N$ and suppose that $m=1$, that is, $N$ is relatively prime to $\delta$.  Then the second case in 
Theorem~\ref{Theorem:Moments} does not apply and hence
\begin{equation*}
\sum_{\substack{\ell \in \Le\multi{n} \\ \ell \equiv i \pmod{N}}} \!\!\!\!\!\!\!\! \ell^p \,\, =\,\,
 \dfrac{p! \, h_p ( \frac{1}{n_1}, \frac{1}{n_2}, \ldots, \frac{1}{n_k} )}{N(k + p - 1)! (n_1 n_2 \cdots n_k)}  n^{k+p-1} + w_i(n),
\end{equation*}
for $i=0,1,2,\ldots,N-1$, in which $w_i(n)$ is a quasipolynomial of degree at most $k+p-2$ with rational-valued coefficients and period dividing $\lcm(n_1,n_2,\ldots,n_k)$.  In particular, factorization lengths are asymptotically equally distributed across all $N$ equivalence classes. 
\end{Example}

\begin{Example}
For a fixed numerical semigroup, this modular equidistribution 
phenomenon occurs for all moduli $N \ge 1$ if and only if $\delta = 1$, since this ensures there is no
prime $p$ such that 
\begin{equation*}
n_1 \equiv n_2 \equiv \cdots \equiv n_k \pmod{p}.
\end{equation*}
This holds, for example, for the McNugget semigroup $\<6,9,20\>$.  As Table~\ref{Table:McNugget} illustrates, the equidistribution across congruence classes modulo $N$ is apparent even for relatively small values of $n$.  The semigroup $\<17, 29, 47, 65\>$, on the other hand, has $\delta = 6$, and the distributions are depicted in Table~\ref{Table:BiggerDelta}.  
\end{Example}

\begin{Example} 
For each prime $p$ and positive integer $k$, there are precisely $p^k-1$ admissible $k$-tuples 
$(n_1,n_2,\ldots,n_k) \pmod{p}$
that occur as $\<n_1,n_2,\ldots,n_k\>$ ranges over all $k$-generator numerical semigroups;
the requirement that $\gcd(n_1,n_2,\ldots,n_k) = 1$ ensures that $n_1 \equiv n_2 \equiv \cdots \equiv n_k \equiv 0 \pmod{p}$ is not possible.
Each of these $p^k-1$ possible $k$-tuples occurs with equal likelihood.  Of these, there are precisely $p-1$ ``bad'' $k$-tuples that yield numerical semigroups for which $\delta \ne 1$; these are the $k$-tuples whose entries are all equal to $i$ for some $1 \leq i \leq p-1$.  Thus, the probability\footnote{The fact that $\NN^k$ does not admit a uniform probability distribution can be remedied by studying numerical semigroups whose generators are at most a given threshold $R$, as in~\cite{arnold,bourgainsinai}.  The infinite products should be replaced by products over primes $p\leq f(R)$, in which $f(R)$ is a suitable function that tends to infinity as $R \to \infty$.  Letting $R \to \infty$ yields the desired result.  Note that other models for randomly selecting numerical semigroups have been studied as well~\cite{RNS}.}
that randomly selected numerical semigroup generators $n_1,n_2,\ldots,n_k$ are not mutually congruent modulo $p$ is
\begin{equation*}
1 - \frac{p-1}{p^k-1} = \frac{(p^k-1)-(p-1)}{p^k-1} = \frac{p^k-p}{p^k-1} = \frac{1 - \frac{1}{p^{k-1} }}{1 - \frac{1}{p^k}}.
\end{equation*}
The Chinese remainder theorem and the Euler product formula 
imply that the probability that
$\<n_1,n_2,\ldots,n_k\>$ fails to satisfy $\delta = 1$ is
\begin{equation*}
\prod_p \Bigg(\frac{1 - \frac{1}{p^{k-1} }}{1 - \frac{1}{p^k}} \Bigg)
= \prod_p \Bigg( \frac{1}{1 - \frac{1}{p^k}} \Bigg) \Bigg\slash\prod_p \Bigg( \frac{1}{1 - \frac{1}{p^{k-1}}} \Bigg) \\
= \frac{ \zeta(k)}{\zeta(k-1)},
\end{equation*}
in which 
\begin{equation*}
\zeta(s) = \sum_{n=1}^{\infty} \frac{1}{n^s}
\end{equation*}
denotes the Riemann zeta function and the products run over all prime numbers.

Since $\lim_{k\to\infty} \zeta(k) = 1$, the probability that a $k$-generator
numerical semigroup satisfies $\delta = 1$ tends to $1$ as $k\to \infty$; see Table \ref{Table:Zeta}.  This is intuitively clear, since the more generators a semigroup has, the more unlikely it is that they will be mutually congruent modulo some prime.  For $k=2$, the pole of $\zeta$ at $1$ ensures that the probability that a $2$-generator numerical semigroup has $\delta = 1$ is $0$.  This makes sense:\ the only way that $\<n_1,n_2\>$ can have $\delta = 1$ is for $n_2 = n_1+1$, and this is highly unlikely.      
\end{Example}

\begin{table}
\begin{equation*}
\begin{array}{c|ccccccccc}
k &  2 & 3 & 4 & 5 & 6 & 7 & 8 & 9 & 10 \\
 \hline
\frac{ \zeta(k)}{\zeta(k-1)} &  0 & 0.7308 & 0.9004 & 0.9581 & 0.9811 & 0.9912 & 0.9958 & 0.9979 & 0.9990 \\
\end{array}
\end{equation*}
\caption{Probability $\zeta(k) / \zeta(k-1)$ that a $k$-generator numerical semigroup has $\delta = 1$.}
\label{Table:Zeta}
\end{table}

\section{Proof of Theorem \ref{Theorem:Moments}}\label{Section:Proof}
Let $S = \<n_1,n_2,\ldots,n_k\>$ with $\gcd(n_1,n_2,\ldots,n_k)=1$, let $N \in \NN$, define $\delta$ as above, and let
$m = \gcd(\delta, N)$ denote the largest divisor of $N$ such that 
$n_1 \equiv n_2 \equiv \cdots \equiv n_k \pmod{ m}$.  Define
\begin{equation}\label{eq:P}
\Lambda^p_{i,N}(n) := \sum_{\substack{\ell \in \Le\multi{n} \\ \ell \equiv i \pmod{N}}} \!\!\!\!\!\!\!\! \ell^p.
\end{equation}
The associated generating function is
\begin{equation}\label{eq:F1}
F(z):= \sum_{n=0}^{\infty} z^n \Lambda^p_{i,N}(n).
\end{equation}
In the computations that follow, $\zeta$ denotes a primitive $N$th root of unity (not to be confused with the Riemann zeta function) and $i$
an index (not to be confused for the imaginary unit).

\subsection{A two-variable generating function}\label{subsection:generating}
Define
\begin{equation*}
f(z,w) := \prod_{i=1}^k \frac{1}{1 - wz^{n_i}},
\end{equation*}
which satisfies
\begin{align}
f(z,w)
&= \prod_{i=1}^k (1 + wz^{n_i} + w^2z^{2n_i} + \cdots) \nonumber \\
&= \sum_{a_1,a_2,\ldots,a_k \geq 0} w^{a_1 + a_2 + \cdots + a_k}\, z^{a_1 n_1 + a_2 n_2 + \cdots + a_k n_k} \nonumber \\
&= \sum_{n=0}^\infty z^n \sum_{\ell=0}^\infty (\text{\# of factorizations of $n$ of length $\ell$}) w^{\ell} \nonumber \\  
&= \sum_{n=0}^\infty z^n \sum_{\ell \in \Le\multi{n}} w^{\ell}. \nonumber
\end{align}
Observe that for $p\in \NN$, 
\begin{equation}\label{eq:fuzz}
\left( \! w \frac{\partial}{\partial w} \! \right)^{\!\!p} \!\! f(z,w) 
=\sum_{n=0}^\infty z^n \sum_{\ell \in \Le[n]}  w^{\ell} \ell^p.
\end{equation}

\subsection{Fourier inversion}
We claim that
\begin{equation}\label{eq:F2}
F(z) = \frac{1}{N} \sum_{j=0}^{N-1} \zeta^{ij} F^p_{j,N}(z),
\end{equation}
in which
\begin{equation}\label{eq:fm}
F^p_{j,N}(z)  := \bigg(\bigg(  w \dfrac{\partial}{\partial w} \bigg)^{\!\!p} f(z,w) \bigg) \bigg|_{w=\overline{\zeta}^{j}} .
\end{equation}
To prove this result, let
\begin{equation}\label{eq:Rmin}
V^p_{i,N}(n) := 
\sum_{j=0}^{N-1} \zeta^{-ij}  \Lambda^p_{j,N}(n)
=\sum_{\ell \in \Le\multi{n}} \zeta^{-i\ell} \ell^p
\end{equation}
and then use Fourier inversion to obtain
\begin{equation}\label{eq:lincomb}
\Lambda^p_{i,N}(n) 
= \frac{1}{N}  \sum_{j=0}^{N-1} \zeta^{ij}  V^p_{j,N}(n). 
\end{equation}
Consequently,
\begin{align*}
F(z)
&= \sum_{n=0}^{\infty} z^n  \Lambda^p_{i,N}(n) && \text{(by \eqref{eq:F1})} \\
&= \sum_{n=0}^{\infty} z^n \left( \frac{1}{N}  \sum_{j=0}^{N-1} \zeta^{ij}  V^p_{j,N}(n) \right) &&\text{(by \eqref{eq:lincomb})} \\
&= \frac{1}{N} \sum_{j=0}^{N-1}  \left( \zeta^{ij}   \sum_{n=0}^{\infty} z^n V^p_{j,N}(n) \right) \\
&= \frac{1}{N} \sum_{j=0}^{N-1}  \left( \zeta^{ij}   \sum_{n=0}^{\infty} z^n \sum_{\ell \in \Le\multi{n}} \zeta^{-j\ell} \ell^p\right) &&\text{(by \eqref{eq:Rmin})}\\
&= \frac{1}{N}\sum_{j=0}^{N-1}    \zeta^{ij} \left(\sum_{n=0}^{\infty} z^n \sum_{\ell \in \Le\multi{n}} (\zeta^{-j})^{\ell} \ell^p \right)\\ 
&= \frac{1}{N}\sum_{j=0}^{N-1}    \zeta^{ij}  \left(\left(  w \dfrac{\partial}{\partial w} \right)^{\!\!p} f(z,w) \right) \bigg|_{w=\overline{\zeta}^{j}} &&\text{(by \eqref{eq:fuzz})}\\
&= \frac{1}{N} \sum_{j=0}^{N-1} \zeta^{ij} F^p_{j,N}(z) &&\text{(by \eqref{eq:fm})}.
\end{align*}

\subsection{Rational representation}\label{subsection:rational}
We now represent $F^p_{j,N}(z)$ as an explicit rational function.
First, we require the identity
\begin{equation}\label{eq:partialidentity}
\frac{\partial^p}{\partial w^p} f(z,w)
= p! \left( \prod_{b=1}^k \frac{1}{1 - wz^{n_b}} \right) h_p \left( \frac{z^{n_1}}{1 - wz^{n_1}}, \cdots, \frac{z^{n_k}}{1 - wz^{n_k}} \right),
\end{equation}
which can be verified by induction \cite[Lem.~22]{GOOY}.  Second, recall that
the \emph{Stirling number of the second kind} $\stirling{n}{i}$ counts the number of partitions of 
$\{1, 2, \ldots, n\}$ into $i$ nonempty subsets. It is known that
\begin{equation}\label{eq:stirlingidentity}
\bigg(x \frac{d}{dx}\bigg)^{\!p}
= \sum_{i=0}^p \stirling{p}{i} x^i \frac{d^i}{dx^i}
\end{equation}
for $p \in \NN$ \cite{Carlitz1, Carlitz2, Toscano}.
Now compute
\begin{align}
F^p_{j,N}(z)
&= \bigg(  \bigg( \! w \frac{\partial}{\partial w} \! \bigg)^{\!\!p} \!\! f(z,w) \bigg) \bigg |_{w=\overline{\zeta}^j} && \text{(by \eqref{eq:fm})}\nonumber\\
&= \sum_{a=0}^p \stirling{p}{a} \, w^a \frac{\partial^a f}{\partial w^a} \, \bigg |_{w=\overline{\zeta}^j} \nonumber && \text{(by \eqref{eq:stirlingidentity})} \\
&= \scalebox{0.9}{$\displaystyle\sum_{a=0}^p \stirling{p}{a} \, a! \, w^a \left( \prod_{b=1}^k \frac{1}{1 - wz^{n_b}} \right) h_a \left( \frac{z^{n_1}}{1 - wz^{n_1}}, \cdots, \frac{z^{n_k}}{1 - wz^{n_k}} \right) \bigg |_{w=\overline{\zeta}^j}$} && \text{(by \eqref{eq:partialidentity})} \nonumber \\
&=  \scalebox{0.9}{$\displaystyle\sum_{a=0}^p \stirling{p}{a} \, a! \,\overline{\zeta}^{aj}
\left( \prod_{b=1}^k \frac{1}{1 - \overline{\zeta}^j z^{n_b}}  \right) h_a \left( \frac{z^{n_1}}{1 - \overline{\zeta}^j z^{n_1}}, \cdots, \frac{z^{n_k}}{1 - \overline{\zeta}^j z^{n_k}} \right)$}.\nonumber
\end{align}

\subsection{A crucial subgroup}
Consider the subgroup
\begin{equation}\label{eq:Gamma}
\Gamma = \{ t \in \ZZ/N\ZZ : n_1 t \equiv n_2 t \equiv \cdots \equiv n_k t \pmod{N}\}
\end{equation}
of $\ZZ/N\ZZ$.  Then $\Gamma$ is cyclic with generator  
$N / |\Gamma|$ and hence
\begin{equation}\label{eq:1w1}
n_1 \equiv n_2 \equiv \cdots \equiv n_k \pmod{|\Gamma|}.
\end{equation}
Now $m = \gcd(N,\delta)$ is the largest divisor of $N$ such that
\begin{equation}\label{eq:1w2}
n_1 \equiv n_2 \equiv \cdots \equiv n_k \pmod{m}.
\end{equation}
Then $N/m$ generates a subgroup $\Gamma'$ of $\ZZ/N \ZZ$ of order $m$.
In particular,
\begin{equation}\label{eq:2w1}
n_1 t \equiv n_2 t \equiv \cdots \equiv n_k t \pmod{N}
\end{equation}
for all $t \in \Gamma'$.  
The maximality of $m$, \eqref{eq:1w1}, and \eqref{eq:1w2} imply that
$|\Gamma| \leq |\Gamma'|$.  
On the other hand, \eqref{eq:Gamma} and \eqref{eq:2w1}
ensure that $\Gamma' \subseteq \Gamma$ and hence $|\Gamma'| \leq |\Gamma|$.
Since $\Gamma, \Gamma'$ are subgroups of $\ZZ/N\ZZ$ of the same order,
$\Gamma = \Gamma'$.  In particular, $|\Gamma| = m$.

\subsection{An automorphism}
The definition \eqref{eq:Gamma} of the group $\Gamma$ ensures that 
multiplication modulo $N$ by any of the numerical semigroup generators $n_1,n_2,\ldots,n_k$ yields the same homomorphism $\alpha:\Gamma \to \Gamma$.
We claim that $\alpha$ is an automorphism.
Since $\gcd(n_1,n_2,\ldots,n_k) =1$, there are $b_1,b_2,\ldots,b_k \in \ZZ$ such that
\begin{equation*}
b_1 n_1 + b_2 n_2 + \cdots + b_k n_k = 1.
\end{equation*}
If $r \in \Gamma$ and $\alpha(t) = r$, then
\begin{equation}\label{eq:nootie}
n_1 t \equiv n_2 t \equiv \cdots \equiv n_k t \equiv r \pmod{N}
\end{equation}
and hence
\begin{align*}
t
&= t(b_1n_1+b_2n_2 + \cdots+b_kn_k) \\
&= b_1 (n_1 t) + b_2 (n_2 t) + \cdots + b_k (n_k t) \\
&\equiv (b_1 + b_2 + \cdots + b_k)r \pmod{N}.
\end{align*}
In particular, the kernel of $\alpha$ is trivial and thus $\alpha$ is an automorphism.
Note that $r,t \in \ZZ/N \ZZ$ satisfy \eqref{eq:nootie} if and only if $r,t \in \Gamma$ and $\alpha(t) = r$.

\subsection{An exponential sum}
Since $\zeta^{ N/m}$ is a primitive $m$th root of unity, 
\begin{align}
\sum_{t \in \Gamma} \zeta^{i\alpha(t)-tn}
&= \sum_{t \in \Gamma} \zeta^{i (n_1 t) -tn} 
= \sum_{t \in \Gamma} \zeta^{t (i n_1 -n)} \nonumber\\
&= \sum_{a=1}^{m } \zeta^{ \frac{aN}{m}(i n_1 -n)}
= \sum_{a=1}^{m} (\zeta^{ \frac{N}{m}})^{a(i n_1 -n)} \nonumber \\
&=
\begin{cases}
m & \text{if $in_1 \equiv n \pmod{ m}$}, \\[5pt]
0 & \text{if $i n_1 \not\equiv n \pmod{ m}$}.
\end{cases}
\label{eq:Finale}
\end{align}

\subsection{Common zeros}
For $0 \leq r \leq N-1$ and $1 \leq i \leq k$, 
the polynomial
\begin{equation*}
\phi_r^i(z):=1-\o{\zeta}^r z^{n_i}
\end{equation*}
has zeros $\zeta^{\frac{r+sN}{n_i}}$ for $1 \leq s \leq n_i$.
These zeros are distinct because $r+sN \equiv r + s'N \pmod{Nn_i} $ implies 
$s \equiv s' \pmod{n_i}$, and hence $s=s'$.

\begin{Lemma}\label{Lemma:Common}
    Fix $r \in \{0,1,\ldots,N-1\}$.  Then
    $\zeta^t$ is a common zero of the polynomials
    \begin{equation}\label{eq:fiji}
        \phi_r^1(z),\,\,\phi_r^2(z),\ldots,\,\,\phi_r^k(z)
    \end{equation}
    if and only if $r,t \in \Gamma$ and $\alpha(t) = r$.
\end{Lemma}

\begin{proof}
    The polynomials \eqref{eq:fiji} have a common zero if and only if there are $s_1,s_2,\ldots,s_k \in \ZZ$ such that 
    \begin{equation*}
        \frac{r+s_1N}{Nn_1} = \frac{r+s_2N}{Nn_2} = \cdots = \frac{r+s_kN}{Nn_k},
    \end{equation*}
    or, equivalently
    \begin{equation}\label{eq:lcm}
        \frac{L}{n_1}(r+s_1N) = \frac{L}{n_2}(r+s_2N) = \cdots = \frac{L}{n_k}(r+s_kN),
    \end{equation}
    in which $L:=\lcm(n_1,n_2,\ldots,n_k)$.  
    Since $\gcd(n_1,n_2,\ldots,n_k)=1$, we have\footnote{This is because a prime power exactly divides $L$ if and only if it divides at least one, but not all, of $n_1,n_2,\ldots,n_k$.
    Consequently, it exactly divides the expression on the right-hand side of \eqref{eq:LLn}.}%
    \begin{equation}\label{eq:LLn}
        L=\lcm \left( \frac{L}{n_1}, \frac{L}{n_2},\ldots,\frac{L}{n_k} \right)
    \end{equation}
    and hence the integer \eqref{eq:lcm} is a multiple of $L$.  Thus,
    the polynomials \eqref{eq:fiji} have a common zero if and only if
    there is a $t \in \ZZ$ such that 
    \begin{equation}\label{eq:eqn}
        r+s_iN=tn_i \quad \text{for all $i=1,2,\ldots,k$};
    \end{equation} 
    that is, if and only if \eqref{eq:nootie} holds.
    This is equivalent to $r,t \in \Gamma$ and $\alpha(t)=r$.
\end{proof}

\subsection{A residue computation}
The maximum possible order for a pole of 
\begin{equation*}
F^p_{j,N}(z) =
\sum_{a=0}^p \stirling{p}{a} \, a! \,\overline{\zeta}^{aj}
\left( \prod_{b=1}^k \frac{1}{1 - \overline{\zeta}^j z^{n_b}}  \right) h_a \left( \frac{z^{n_1}}{1 - \overline{\zeta}^j z^{n_1}}, \cdots, \frac{z^{n_k}}{1 - \overline{\zeta}^j z^{n_k}} \right)
\end{equation*}
is $k+p$, which can only arise from the summand with $a = p$.
Lemma \ref{Lemma:Common} ensures that $\zeta^t$ is a pole of $F^p_{r,N}$
with order $k+p$ if and only if $r,t \in \Gamma$ and $\alpha(t) = r$.  In particular,
\begin{equation}\label{eq:1zkmFmjz}\qquad
\lim_{z\to \zeta^t} (1-\overline{\zeta}^t z)^{k+p} F^p_{j,N}(z)  = 0,\qquad \text{for $j \neq r$}.
\end{equation}

From \eqref{eq:F2}, we see that $F$
has a pole of order $k+p$ at $\zeta^t$ for each $t \in \Gamma$
and these are the only poles of $F$ of this (maximal) degree.
Write
\begin{equation}\label{eq:Fcg}
F(z) = \sum_{t \in \Gamma}\left( \frac{C_t}{(1 - \overline{\zeta}^t z)^{k+p}} \right)+ G(z),
\end{equation}
in which $C_t \neq 0$ and $G(z)$ is a rational function, all of whose poles are
$L$th roots of unity with order at most $k + p - 1$.  In particular,
\begin{equation*}
G(z) = \sum_{n=0}^\infty u(n) z^n
\end{equation*}
for some quasipolynomial $u(n)$ of degree at most $k + p - 2$ and whose period divides $L$.   Moreover, for each $t \in \Gamma$ we have
\begin{equation*}
\lim_{z \to \zeta^t} (1 - \overline{\zeta}^t z)^{k+p} G(z) = 0.
\end{equation*}
With $r = \alpha(t)$, we have
\begin{align*}
C_t
&= \lim_{z \to \zeta^t} (1-\overline{\zeta}^t z)^{k+p} F(z) \nonumber \\
&=\lim_{z \to \zeta^t} (1-\overline{\zeta}^t z)^{k+p} \left(  \frac{1}{N} \sum_{j=0}^{N-1} \zeta^{ij} F^p_{j,N}(z) \right) && \text{(by \eqref{eq:F2})} \nonumber \\
&= \frac{1}{N}   \sum_{j=0}^{N-1} \zeta^{ij}  \Big(\lim_{z \to \zeta^t} (1-\overline{\zeta}^t z)^{k+p} F^p_{j,N}(z) \Big)\nonumber  \\
&= \frac{\zeta^{ir}}{N}   \Big(\lim_{z \to \zeta^t} (1-\overline{\zeta}^t z)^{k+p} F_{r,N}^p(z) \Big)  &&\text{(by \eqref{eq:1zkmFmjz})} \nonumber \\
&= \scalebox{0.8}{$\displaystyle\frac{\zeta^{ir}}{N}\lim_{z \to \zeta^t}\left[    \sum_{a=0}^p \stirling{p}{a} \, a! \,\overline{\zeta}^{ar}\,(1 - \overline{\zeta}^t z)^{k+p}
\left( \prod_{b=1}^k \frac{1}{1 - \overline{\zeta}^r z^{n_b}}  \right) h_a \left( \frac{z^{n_1}}{1 - \overline{\zeta}^r z^{n_1}}, \cdots, \frac{z^{n_k}}{1 - \overline{\zeta}^r z^{n_k}} \right) \right]$ } \\
&= \scalebox{0.925}{$\displaystyle\frac{\zeta^{ir} \overline{\zeta}^{pr}p!}{N}\lim_{z \to \zeta^t} (1 - \overline{\zeta}^t z)^{k+p} \left( \prod_{b=1}^k \frac{1}{1 - \overline{\zeta}^r z^{n_b}} \right) h_p \left( \frac{z^{n_1}}{1 - \overline{\zeta}^r z^{n_1}}, \cdots, \frac{z^{n_k}}{1 - \overline{\zeta}^r z^{n_k}} \right) $} \\
&= \frac{\zeta^{ir} \overline{\zeta}^{pr}p!}{N}\lim_{z \to \zeta^t} \left( \prod_{b=1}^k \frac{1 - \overline{\zeta}^t z}{1 - \overline{\zeta}^r z^{n_b}} \right) h_p \left( z^{n_1}\frac{1 - \overline{\zeta}^t z}{1 - \overline{\zeta}^r z^{n_1}}, \cdots, z^{n_k}\frac{1-\overline{\zeta}^t z}{1 - \overline{\zeta}^r z^{n_k}} \right) \\
&= \frac{\zeta^{ir} \overline{\zeta}^{pr}p!}{N}\left( \prod_{b=1}^k \frac{\zeta^r}{\zeta^{n_b t}n_b}  \right) h_p \left( \frac{\zeta^r}{n_1},\ldots, \frac{\zeta^r}{n_k} \right)
&& (\text{L'H\^opital}) \nonumber \\
&= \frac{\zeta^{ir}  p!}{N}\left( \prod_{b=1}^k \frac{\zeta^r}{\zeta^{n_b t}n_b}  \right) h_p \left( \frac{1}{n_1},\ldots, \frac{1}{n_k} \right) \nonumber 
&& (\text{Def.~of $h_p$})\\
&= \frac{\zeta^{i \alpha(t)}  p!}{N(n_1 n_2 \cdots n_k)} h_p \left( \frac{1}{n_1},\ldots, \frac{1}{n_k} \right) \nonumber 
&& (\text{by \eqref{eq:nootie}}).
\end{align*}
Here are the details for the somewhat involved L'H\^opital step:
\begin{equation*}
\lim_{z \to \zeta^t}z^{n_i}  \frac{(1 - \overline{\zeta}^t z)}{1 - \overline{\zeta}^r z^{n_i}}
= \zeta^{n_i t} \lim_{z \to \zeta^t} \frac{1 - \overline{\zeta}^t z}{1 - \overline{\zeta}^r z^{n_i}} 
= \zeta^{n_i t} \frac{\overline{\zeta}^t}{n_i \overline{\zeta}^r (\zeta^t)^{n_i-1}} 
= \zeta^{n_i t} \frac{\overline{\zeta}^t \zeta^t}{n_i \overline{\zeta}^r \zeta^{n_i t}} 
= \frac{\zeta^r}{n_i}.
\end{equation*}

\subsection{Conclusion}

Now observe that
\begin{align*}
\frac{1}{(1 - \overline{\zeta}^t z)^{k+p}}
&= \sum_{n=0}^\infty \binom{n + k + p - 1}{k + p - 1} \overline{\zeta}^{tn}z^n \\
&= \sum_{n=0}^\infty \frac{(n + k + p - 1) \cdots (n + 1)}{(k + p - 1)!} \overline{\zeta}^{tn}z^n \\ 
&= \frac{1}{(k + p - 1)!} \sum_{n=0}^\infty \big(n^{k+p-1} + v(n)\big) \overline{\zeta}^{tn} z^n,
\end{align*}
in which $v(n)$ is a quasipolynomial of degree $k + p - 2$ with integer coefficients.
Our recent evaluation of $C_t$ and \eqref{eq:Fcg} imply
\begin{align*}
F(z)
&=\sum_{t \in \Gamma}\left( \frac{C_t}{(1 - \overline{\zeta}^t z)^{k+p}} \right)+ G(z)\\ 
&=\sum_{t \in \Gamma} \left[\frac{\zeta^{i\alpha(t)} p!}{N (n_1 n_2\cdots n_k)} h_p \left( \frac{1}{n_1}, \frac{1}{n_2}, \ldots, \frac{1}{n_k} \right)  \frac{1}{(1 - \overline{\zeta}^t z)^{k+p}} \right]+ G(z)\\
&= \frac{ p!h_p ( \frac{1}{n_1}, \ldots, \frac{1}{n_k} )}{N (n_1 n_2\cdots n_k)} 
\left(\sum_{t\in \Gamma}\frac{\zeta^{i\alpha(t)}}{(1 - \overline{\zeta}^t z)^{k+p}} \right) + G(z)\\
&= \frac{ p!h_p ( \frac{1}{n_1}, \ldots, \frac{1}{n_k} )}{N (n_1 n_2\cdots n_k)} 
\left(\sum_{t\in \Gamma} \zeta^{i\alpha(t)} \sum_{n=0}^{\infty} \big(n^{k+p-1} + v(n)\big) \overline{\zeta}^{tn}z^n\right) + \sum_{n=0}^{\infty} u(n)z^n\\
&=\frac{p! h_p ( \frac{1}{n_1}, \ldots, \frac{1}{n_k} )}{N(k + p - 1)! (n_1 \cdots n_k)}
\sum_{n=0}^{\infty}  (n^{k+p-1}+v(n)) \Big(\sum_{t \in \Gamma} \zeta^{i\alpha(t)-tn}\Big)z^n + \sum_{r\in \Gamma} u(z)z^n.
\end{align*}

The evaluation \eqref{eq:Finale} of the parenthesized exponential sum 
and the definition \eqref{eq:F1} of $F$ as the generating function 
for $\Lambda^p_{i,N}(n)$ reveal that
\begin{equation*}
\Lambda^p_{i,N}(n) 
=
\begin{cases}
\dfrac{p! \,m\, h_p ( \frac{1}{n_1}, \frac{1}{n_2}, \ldots, \frac{1}{n_k} )}{N(k + p - 1)! (n_1 n_2 \cdots n_k)}  n^{k+p-1} + w_i(n),
&\text{if $n \equiv in_1 \pmod{m}$},\\[8pt]
0 & \text{if $n \not\equiv in_1 \pmod{m}$},
\end{cases}
\end{equation*}
in which $w_i(n)$ is a quasipolynomial of degree at most $k+p-2$ whose coefficients have period dividing 
$N\lcm(n_1,n_2,\ldots,n_k)$.  Since $u(n)$ and $v(n)$ both have rational coefficients, so must $w(n)$.  
This concludes the proof. \qed

\section{Conclusion}\label{Section:Conclusion}
While this paper largely settles the matter of asymptotic modular distribution of factorization lengths 
for elements in numerical semigroups, a related question worthy of further research remains.
Can one characterize the rate of convergence in Theorem \ref{Theorem:Probability}?
This would, presumably, require a detailed examination of the quasipolynomial error term $w_i(n)$ 
in Theorem \ref{Theorem:Moments}
and its dependence upon $n_1,n_2,\ldots,n_k$, $N$, $i$, and $\delta$, along with the congruence class of $n$
modulo $N$.  
A careful study of the proof of Theorem \ref{Theorem:Moments}
might yield sufficiently explicit bounds upon the $w_i(n)$ to carry this out.

\bibliography{FLD4AS3-Modular}
\bibliographystyle{amsplain}

\end{document}